\documentclass[12pt]{amsart}

\usepackage{graphicx}
\usepackage{amsmath}
\usepackage{amscd}
\usepackage{amsfonts}
\usepackage{amssymb}
\usepackage{color}

\usepackage[unicode=true,
 bookmarks=true,bookmarksnumbered=false,bookmarksopen=false,
 breaklinks=false,pdfborder={0 0 1},backref=false,colorlinks=false]
 {hyperref}
\hypersetup{pdftitle={Essential normality, essential norms and hyperrigidity},
 pdfauthor={Matthew Kennedy and Orr Shalit}}

\numberwithin{equation}{section}

\newtheorem{theorem}{Theorem}[section]
\newtheorem{lemma}[theorem]{Lemma}
\newtheorem{proposition}[theorem]{Proposition}

\newtheorem{definition}[theorem]{Definition}
\newtheorem{corollary}[theorem]{Corollary}

\theoremstyle{definition}
\newtheorem{example}[theorem]{Example}
\newtheorem{remark}[theorem]{Remark}

\newcommand{\be}{\begin{equation}}
\newcommand{\ee}{\end{equation}}
\newcommand{\bes}{\begin{equation*}}
\newcommand{\ees}{\end{equation*}}

\newcommand{\cC}{\mathcal{C}}

\newcommand{\cK}{\mathcal{K}}

\newcommand{\cF}{\mathcal{F}}

\newcommand{\cA}{\mathcal{A}}
\newcommand{\cB}{\mathcal{B}}

\newcommand{\cO}{\mathcal{O}}
\newcommand{\cS}{\mathcal{S}}
\newcommand{\cT}{\mathcal{T}}

\newcommand{\cZ}{\mathcal{Z}}

\newcommand{\lel}{\left\langle}
\newcommand{\rir}{\right\rangle}

\newcommand{\mb}[1]{\mathbb{#1}}

\newcommand{\Bd}{\mathbb{B}_d}
\newcommand{\pBd}{\partial \mathbb{B}_d}
\newcommand{\pV}{\partial V}
\newcommand{\oV}{\overline{V}}

\newcommand{\rank}{\operatorname{rank}}

\newcommand{\alg}{\operatorname{Alg}}

\newcommand{\Spec}{\operatorname{Spec}}
\newcommand{\spn}{\operatorname{span}}
\newcommand{\lip}{\langle}
\newcommand{\rip}{\rangle}

\begin{document}

\title{Essential normality, essential norms and hyperrigidity}

\author{Matthew Kennedy} 
\address{School of Mathematics and Statistics \\
Carleton University \\
1125 Colonel By Drive \\
Ottawa, Ontario K1S 5B6 \\
Canada}
\email{mkennedy@math.carleton.ca}

\author[Orr Shalit]{Orr Moshe Shalit} 
\address{Department of Mathematics\\
Faculty of natural Sciences\\
Ben-Gurion University of the Negev\\
Be'er Sheva\; 84105\\
Israel}
\email{oshalit@math.bgu.ac.il}

\thanks{First author supported by a grant from the Natural Sciences and Engineering Research Council of Canada. The second author is partially supported by ISF Grant no. 474/12, by 
EU FP7/2007-2013 Grant no. 321749, and by GIF Grant no. 2297-2282.6/20.1. }

\begin{abstract}
Let $S = (S_1, \ldots, S_d)$ denote the compression of the $d$-shift to the complement of a homogeneous ideal $I$ of $\mb{C}[z_1, \ldots, z_d]$. Arveson conjectured that $S$ is essentially normal. In this paper, we establish new results supporting this conjecture, and connect the notion of essential normality to the theory of the C*-envelope and the noncommutative Choquet boundary. 

The unital norm closed algebra $\cB_I$ generated by $S_1,\ldots,S_d$ modulo the compact operators is shown to be completely isometrically isomorphic to the uniform algebra generated by polynomials on $\overline{V} := \overline{\cZ(I) \cap \Bd}$, where $\cZ(I)$ is the variety corresponding to $I$. Consequently, the essential norm of an element in $\cB_I$ is equal to the sup norm of its Gelfand transform, and the C*-envelope of 
$\cB_I$ is identified as the algebra of continuous functions on $\overline{V} \cap \pBd$, which means it is a complete invariant of the topology of the variety determined by $I$ in the ball.

Motivated by this determination of the C*-envelope of $\cB_I$, we suggest a new, more qualitative approach to the problem of essential normality. We prove the tuple $S$ is essentially normal if and only if it is hyperrigid as the generating set of a C*-algebra, which is a property closely connected to Arveson's notion of a boundary representation.

We show that most of our results hold in a much more general setting. In particular, for most of our results, the ideal $I$ can be replaced by an arbitrary (not necessarily homogeneous) invariant subspace of the $d$-shift.
\end{abstract}

\subjclass[2010]{47A13, 47L30, 46E22}
\keywords{Essentially normal operators, Hilbert modules, Non-selfadjoint operator algebras, C*-envelope, Drury-Arveson space, the $d$-shift, von Neumann inequality}

\maketitle

\section{Introduction, notation and preliminaries} 

The purpose of this paper is to collect evidence supporting Arveson's conjecture on essential normality, and to connect the conjecture with the theory of the C*-envelope and the noncommutative Choquet boundary. Our results are of a nature quite different from other results on this conjecture, e.g., \cite{Arv05,Arv07,Dou06a,Dou06b,DS11,DW12,Esc11,GuoWang,Ken12,KennedyShalit12,Sha11}; these previous results gave a full verification of the conjecture for limited classes of (typically homogeneous) ideals. Here, we shall present more limited results that hold for all homogeneous ideals, and for a large class of non-homogeneous ideals. Arveson's conjecture has several interesting and non-trivial consequences. We shall prove some of these consequences directly, thereby gathering evidence supporting the conjecture. 

This work also connects to the ongoing effort to understand operator algebras arising from {\em subproduct systems} (see \cite{DRS,Dor-On,Gurevich,Har12,ShalitSolel,ViselterCovariant, ViselterCuntz}). If we restrict attention to homogeneous ideals, then the algebras studied in this paper are precisely the algebras arising from commutative subproduct systems over $\mb{N}$, with finite dimensional Hilbert spaces as fibres. 

\subsection{Preliminaries}
Throughout, $d\geq 2$ is a fixed integer, $\mb{B}_d$ denotes the unit ball in $\mb{C}^d$, and $\mb{C}[z] = \mb{C}[z_1, \ldots, z_d]$ denotes the algebra of complex polynomials in $d$ variables. Let $E$ be a Hilbert space with orthonormal basis $\{e_1, \ldots, e_d\}$. Then we may identify the symmetric tensor algebra over $E$ with $\mb{C}[z]$. 
Form the symmetric Fock space over $E$:
\[
\cF^+(E) = \mb{C} \oplus E \oplus E^2 \oplus \ldots 
\]
The space $\cF^+(E)$ is also called the Drury--Arveson space. It can be naturally identified as the reproducing kernel Hilbert space on the unit ball with reproducing kernel
\[
k_w(z) = \frac{1}{1 - \lel z, w \rir}, \quad w,z \in \Bd. 
\]
In this function-theoretic incarnation the Drury-Arveson space is usually denoted by $H^2_d$, and we shall use this notation here. A third, equivalent, way of viewing this space is simply as the completion of $\mb{C}[z]$ under the inner product that makes monomials orthogonal and assigns to each monomial the norm 
\[
\|z^{\alpha}\|^2 = \frac{\alpha_1! \cdots \alpha_d!}{(\alpha_1 + \ldots +\alpha_d)!} ,
\]
for $\alpha = (\alpha_1, \ldots, \alpha_d) \in \mb{N}^d$ and $z^\alpha = z_1^{\alpha_1} \cdots z_d^{\alpha_d}.$

There is a natural row contraction acting on $H^2_d$, namely the $d$-shift $M_z = (M_{z_1}, \ldots, M_{z_d})$, where for $i=1,\ldots,d$, $M_{z_i}$ denotes the operator of multiplication by $z_i$ (here and below, by \emph{row contraction} we mean a tuple of operators $T = (T_1, \ldots, T_d)$ satisfying $\sum T_i T_i^* \leq I$). The $d$-shift enjoys some universal properties which make it particularly worthy of  investigation (see the recent survey \cite{ShalitSurvey}). 

Let $I \triangleleft \mb{C}[z]$ be an ideal. 
In the literature, most attention has been paid to the case where $I$ is a homogeneous ideal, but we shall require less. 
Denote by $\cZ(I)$ the zero set of the ideal $I$, let $\overline{V}(I) := \overline{\cZ(I) \cap \mb{B}_d}$ and 
\bes
\pV(I) = \overline{V}(I) \cap \partial \mb{B}_d .
\ees
Below, when $I$ will be fixed, we will write simply $\pV = \pV(I)$ and $\overline{V} = \overline{V}(I)$. We make the following standing assumptions on the ideal $I$. 
\be\label{eq:stand_assum}
\overline{V}(I) := \overline{\cZ(I) \cap \mb{B}_d} = \cZ(I) \cap \overline{\mb{B}_d},
\ee
and
\be\label{eq:stand_assum2}
\pV \neq \emptyset . 
\ee
The above assumptions imply, in particular, $I$ is not of finite co-dimension in $\mb{C}[z]$. 
The assumptions (\ref{eq:stand_assum}, \ref{eq:stand_assum2}) are clearly satisfied by every homogeneous (or more generally, every quasi-homogeneous) ideal of infinite co-dimension. 

Form the space
\be\label{eq:Fock}
\cF_I = H^2_d \ominus I .
\ee
The operator $d$-tuple of interest in this paper is the compression $S = (S_1,\ldots,S_d)$ of the $d$-shift to $\cF_I$. That is, the operators $S_1, \ldots, S_d$ are defined by 
\bes
S_i = P_{\cF_I} M_{z_i} \big|_{\cF_{I}} \,\, , \,\, i=1, \ldots, d.
\ees
Note that for every $p \in I$, we have that $p(S) = 0$. 
The $d$-tuple $S$ depends on $I$, but since the choice of $I$ will always be clear from the context, our choice of notation will not reflect this.

Circa 2000, Arveson conjectured that for $1 \leq i,j \leq d$ the commutator $S_i S^*_j - S^*_j S_i$ is compact, or in other words, that $S$ is essentially normal. 
(The up-to-date version of the conjecture, due to Arveson \cite{Arv02} and refined by Douglas \cite{Dou06b}, is that $S$ is $p$-essentially normal, meaning that $|S_i S^*_j - S^*_j S_i|^p$ is trace class for all $p > \dim \cZ(I)$.) 
Before presenting our results, it will be convenient to introduce a few more pieces of notation.

We define $\cA_I^o$ to be the unital algebra generated by $S_1, \ldots, S_d$, and we define $\cA_I$ to be the norm closure of $\cA^o_I$. The {\em Toeplitz algebra of $I$}, which we denote by $\cT_I$, is defined to be the C*-algebra generated by $\cA_I$. Let $\cK$ denote the compact operators on $\cF_I$. 
It is known (see \cite[Theorem 1.3]{Popescu06} or \cite[Proposition 8.1]{ShalitSolel}) that $\cK \subseteq \cT_I$; thus we may define the {\em Cuntz algebra  of $I$}, which we denote by $\cO_I$, to be the quotient $\cO_I = \cT_I / \cK$.

For the special case where $I = \{0\}$, we write $\cA_d$ for the unital norm closed algebra generated by $S = M_z$. This algebra is the universal operator algebra generated by a commuting row contraction. 
When $I$ is homogeneous, $\cA_I$ is the universal operator algebra generated by a $d$-contraction subject to the relations in $I$ (see \cite[Theorem 8.4]{ShalitSolel}). In this case, we have a natural completely isometric isomorphism
\bes
\cA_I \cong \cA_d / \overline{I}.
\ees

The spectrum of $\cA_I$, that is, the space of characters (non-zero multiplicative linear functionals) on $\cA_I$, may be homeomorphically identified with $\overline{V}$, via the identification of a point with the evaluation functional at that point:
\be\label{eq:spec1}
\Spec(\cA_I) \ni \rho \longleftrightarrow (\rho(S_1), \ldots, \rho(S_d)) \in \overline{V},
\ee
\be\label{eq:spec2}
\overline{V} \ni \lambda = (\lambda_1, \ldots, \lambda_d) \longleftrightarrow \rho_\lambda : S_i \mapsto \lambda_i .
\ee
For the case when $I$ is homogeneous, this was explained in \cite[Section 10.2]{ShalitSolel} and \cite[Section 4.1]{DRS} in detail. 
In general, this fact follows from \cite[Theorem 2.1]{Popescu06}. By that theorem, every point in $\cZ(I) \cap \mb{B}_d$ gives rise to a character, as it is a pure row contraction satisfying the relations in $I$. 
Since the spectrum is closed, we find that every point in $\overline{V}$ also gives rise to a character (it is here that we use the standing assumption (\ref{eq:stand_assum})). 
Conversely, since $p(S) = 0$ for every $p \in I$, and since a multiplicative linear functional on an operator algebra is automatically completely contractive, we conclude that for every character $\rho$, the point $ (\rho(S_1), \ldots, \rho(S_d))$ lies in $\overline{V}$. 

\subsection{Overview of the paper}
The simplest form of Arveson's conjecture can be restated as the assertion that $\cO_I = C(\pV)$. We will obtain some results regarding the structure of $\cO_I$, and the operator algebraic structure of the image of $\cA_I$ in the quotient $\cO_I = \cT_I / \cK$. In Section \ref{sec:C*spectra} we prove that the space of $1$-dimensional representations is homeomorphic to $\pV$, and that the same is true for $\cT_I$. Thus, $\pV$ is an invariant of the operator algebraic structure of $\cO_I$, as expected. 

Arveson's conjecture implies there is a C*-algebra that completely captures the topology of $\pV$, namely $\cO_I$. The existence of a canonical and naturally occurring C*-algebra that encodes the topology of $\pV$ is also suggested by the results of \cite{DRS,Har12}, which showed that, if $I$ is radical and homogeneous, then certain operator algebraic constructs associated with $S$ reflect corresponding geometric properties of $\cZ(I)$. To be precise, it was shown that if $I$ and $J$ are two radical homogeneous ideals, then 
\begin{enumerate}
\item $\cA_I^o$ is algebraically isomorphic to $\cA_J^o$ if and only if $\cZ(I)$ and $\cZ(J)$ are isomorphic in the sense of algebraic geometry;
\item $\cA_I$ is isometrically isomorphic to $\cA_J$ if and only if $\cZ(I)$ is the image of $\cZ(J)$ under a unitary transformation;
\item $\cA_I$ is algebraically isomorphic to $\cA_J$ if and only if the structure of $\cZ(I)$ and $\cZ(J)$ are the same, in an intermediate geometry, which is finer then the algebraic geometry and coarser than the Hilbert space geometry ($\cZ(I)$ is the image of $\cZ(J)$ under a linear map which is length preserving on $\cZ(J)$). 
\end{enumerate}

With these results in mind, it is natural to ask  for a C*-algebra associated with $S$ that completely captures the topology of $\pV$. 

In Section \ref{sec:main} we show that the sought after C*-algebra is the C*-envelope of the unital operator algebra generated by $S$ modulo the compact operators (for an up-to-date introduction to the C*-envelope see \cite{Kakariadis}). Specifically, let $Z = (Z_1, \ldots, Z_d)$ denote the image of $S$ in $\cO_I$. We prove that $C^*_e(\overline{\alg}(Z)) = C(\pV)$, as anticipated by Arveson's conjecture. To obtain this result, we compute the essential norm of any element in $M_n(\cA_I)$, and show that it is given as the sup norm of its Gelfand transform. Consequently, we obtain that $\overline{\alg}(Z)$ is completely isometrically isomorphic to $A(\oV)$ --- the norm closure of polynomials on $\oV$. Section \ref{sec:main} closes with a discussion of a new approach to the essential normality conjecture involving the  C*-envelope, and gives some examples showing the difficulties involved. 

In Section \ref{sec:hyperrigidity}, we connect the notion of essential normality with the theory of the non-commutative Choquet boundary. A fundamental insight due to Arveson is the special significance of $*$-representations of $\cT_I$ that restrict to maps on $\cA_I$ which are rigid, in the sense that they have a unique unital completely positive extension to $\cT_I$. The irreducible $*$-representations of $\cT_I$ with this property are precisely the boundary representations with respect to $\cA_I$ (see \cite{Arv69}, \cite{Arv07} and \cite{DK13}). In \cite{Arv11}, Arveson studied a phenomenon he calls \emph{hyperrigidity}, which occurs when every $*$-representation has this rigidity property. We prove that the essential normality of the $d$-tuple $S$ is equivalent to the statement that the algebra $\cA_I$ is hyperrigid. 

In Section \ref{sec:other-spaces}, we discuss the fact that many of our results hold when the space $H^2_d$ is replaced with an arbitrary Hilbert module satisfying certain natural conditions. In particular, we prove that many of the results hold in the setting of the Besov-Sobolev spaces $B^2_\sigma(\mb{B}_d)$, for $\sigma \geq 1/2$. 

It is well known \cite{Arv98, DavPittsPick} that von Neumann's inequality fails for row contractions, meaning that it is {\em not true} that for every row contraction $T$ and for every polynomial $p$ one has 
\be\label{eq:vN}
\|p(T)\| \leq \sup_{z \in \pBd} |p(z)| .
\ee
In the final section we show that for finite rank row contractions corresponding to homogeneous ideals, von Neumann's inequality holds modulo the compacts. In fact, we show that if $T$ is a commuting row contraction of {\em finite rank}, and if $q(T) = 0$ for every polynomial $q$ lying in some homogeneous ideal $I \triangleleft \mb{C}[z]$ then 
\[
\|p(T)\|_e \leq \sup_{z \in \partial V(I)}|p(z)|
\]
for all $p \in \mb{C}[z]$ (and similarly for matrix valued polynomials). In the case where $T$ does not satisfy any homogeneous polynomial equation, then by setting $I = (0)$, we obtain (\ref{eq:vN}) with the essential norm replacing the norm (this result also follows from the results in \cite{Arv98}). 

\section{C*-algebraic spectra}\label{sec:C*spectra}

The purpose of this section is to determine the $1$-dimensional representations of $\cT_I$ and $\cO_I$. 
We continue to use the notation introduced in the introduction, and we recall the standing assumptions that  (\ref{eq:stand_assum}, \ref{eq:stand_assum2}), and the fact that $\cT_I$ contains the algebra of compact operators $\cK$.

If $\cC$ is a (not-necessarily commutative) C*-algebra, we denote by $\Spec(\cC)$ the space of multiplicative linear functionals on $\cC$, endowed with the weak* topology. Elements in $\Spec(\cC)$ are called {\em characters}.

\begin{proposition}\label{prop:spectra}
$\Spec(\cT_I) \cong \pV$.
\end{proposition}
\begin{proof}
Let $\phi \in \Spec(\cT_I)$. Then by (\ref{eq:spec1}) and (\ref{eq:spec2}) we have that  $\phi \big|_{\cA_I}$ is an evaluation functional $\rho_\lambda$ for some $\lambda \in \overline{V}$. This gives rise to a continuous mapping of $\Spec(\cT_I)$ into $\overline{V}$. It is clear that this mapping is injective, as $S_1, \ldots, S_d$ generate $\cT_I$. 

To see that this mapping is into $\partial V$, let $\phi \big|_{\cA_I} = \rho_\lambda$. Since $I_{\cF_I} - \sum S_i S_i^* = P_\mb{C} \in \cK$ (where $\mb{C}$ is interpreted as the first summand in (\ref{eq:Fock})), and since every character annihilates the compacts, we get 
\bes
\|\lambda \| = \sum |\lambda_i|^2 = 1. 
\ees

It remains to show that every evaluation functional $\rho_\lambda$ on $\cA_I$, with $\lambda \in \pV$, extends to a character of $\cT_I$. Fix $\lambda \in \pV$. By Arveson's Extension Theorem, the contractive map $\rho_\lambda : \cA_I \rightarrow \mb{C}$ extends to a state $\hat{\rho}_\lambda: \cT_I \rightarrow \mb{C}$. This much is true for every $\lambda \in \overline{V}$, but we will show that when $|\lambda| = 1$ then $\hat{\rho}_\lambda$ has to be multiplicative (and consequently that $\hat{\rho}_\lambda$ is a boundary representation for $\cA_I$).  

Let $(\sigma,H,P_L)$ be a minimal Stinespring dilation for $\hat{\rho}_\lambda$. Thus $L \subseteq H$ is a one dimensional subspace of a Hilbert space $H$, $P_L : H \rightarrow L$ is the orthogonal projection onto $L$, and $\sigma : \cT_I \rightarrow B(H)$ is a $*$-representation such that 
\bes
\hat{\rho}_\lambda (a) = P_L \sigma(a) P_L 
\ees
for all $a \in \cT_I$. Moreover, minimality means that 
\be\label{eq:span}
\overline{\spn}\{\sigma(a) \xi : a \in \cT_I, \xi \in L\} = H .
\ee 
Now, $\rho_\lambda = \hat{\rho}_\lambda \big|_{\cA_I}$ is multiplicative. By a well-known principle due to Sarason (Lemma 0 in \cite{Sarason}), $H$ decomposes as $H = H_1 \oplus L \oplus H_2$ such that with respect to this decomposition 
\bes
\sigma(a) = \begin{pmatrix}
  * & 0 & 0 \\
  * & \rho_\lambda(a) & 0 \\
  * & * & * \\ 
\end{pmatrix} 
\ees
for $a \in \cA_I$. In particular, we obtain for $i=1, \ldots, d$ 
\bes
\sigma(S_i) = \begin{pmatrix}
  a_i & 0 & 0 \\
  b_i & \lambda_i & 0 \\
  c_i & d_i & e_i \\ 
\end{pmatrix} .
\ees
From $\sum S_i S_i^* \leq I$ we obtain 
\bes
\sum b_i b_i^* + \|\lambda\|^2 = P_L \sum \sigma(S_i) \sigma(S_i)^* P_L \leq P_L I_H P_L = I_L = 1,
\ees
so $b_i = 0$ for all $i$. Thus $H$ decomposes as $H = L \oplus H_0$, and with respect to this decomposition 
\be\label{eq:matrix}
\sigma(S_i) = \begin{pmatrix}
  \lambda_i & 0 \\
  x_i & y_i \\
\end{pmatrix} \,\, , \,\, i = 1, \ldots, d.
\ee
There is no loss of generality in assuming that $\lambda = (1,0,\ldots, 0)$ (otherwise, apply a unitary to the orthonormal basis appearing in the construction of $\cA_I$ in the opening paragraphs). From (\ref{eq:matrix}) we obtain (since $\|\sigma(S_1)\| \leq 1$)
\bes
\sigma(S_1) = \begin{pmatrix}
  1 & 0 \\
  0 & y_1 \\
\end{pmatrix} .
\ees
and
\bes
\sigma(S_i) = \begin{pmatrix}
  0 & 0 \\
  x_i & y_i \\
\end{pmatrix} \,\, , \,\, i = 2, \ldots, d.
\ees
We will show that $y_1 = I_{H_0}$. This will complete the proof, because it would follow (as above) that $x_i,y_i$ are all zero for $i=2, \ldots, d$. Hence $L$ is reducing for $\sigma$, and because $\sigma$ is a minimal dilation we must have $H_0 = 0$ and $\hat{\rho}_\lambda = \sigma$, so it is a multiplicative functional. 

To show that $y_1 = I_{H_0}$ we proceed as follows. From (\ref{eq:span}) and (\ref{eq:matrix}) we find that 
\begin{align*}
&H_0 = \\
& \overline{\spn}\{p_1(y)p_2(y)^*\cdots p_{n-1}(y)^*p_n(y) x_i \Big|  i = 2, \ldots , d; p_1, \ldots, p_n  \mbox{ polynomials}  \}.
\end{align*}
It suffices to show that every element of the form 
\[
p_1(y)p_2(y)^*p_3(y)\cdots p_{n-1}(y)^*p_n(y) x_i
\]
is invariant under $y_1$. We will make repeated use of the basic fact that if $T$ is a contraction, then $Th = h \Leftrightarrow T^*h = h$. Plugging the matrices (\ref{eq:matrix}) in the identity $\sigma(S_i)\sigma(S_1) = \sigma(S_1) \sigma(S_i)$ we obtain 
\bes
y_1 x_i = x_i \,\,\,\, \textrm{ and } \,\,\,\, y_1 y_i = y_i y_1 . 
\ees
If $p(y)$ is a polynomial in $y_1, \ldots y_d$, then
\bes
y_1 p(y) x_i = p(y) y_1 x_i = p(y) x_i.
\ees
Since $\|y_1\| \leq 1$, we have $y_1^* p(y) x_i = p(y) x_i$ for all $i$. Now if $p$ and $q$ are polynomials, then
\bes
y_1^* q(y)^* p(y) x_i = q(y)^* y_1^* p(y) x_i = q(y)^* p(y) x_i,
\ees
and it follows that $y_1 q(y)^* p(y) x_i = q(y)^* p(y) x_i$. Continuing this way, we find that all elements of the form
\[
p_1(y)p_2(y^*)p_3(y)\cdots p_{n-1}(y)^*p_n(y) x_i
\] 
are invariant for $y_1$, thus $y_1 = I_{H_0}$. 
\end{proof}

We record a corollary of the above proof:
\begin{corollary}\label{cor:extension}
The character $\rho_\lambda$ of $\cA_I$ extends to character $\hat{\rho}_\lambda$ of $\cT_I$ if and only $\lambda \in \pV$. 
\end{corollary}

Let $J$ denote the commutator ideal $\textrm{Comm}(\cT_I)$ of $\cT_I$. Then $J$ is the smallest ideal in $\cT_I$ having a commutative quotient. It is a basic fact that $J$ is equal to the intersection of the kernels of all characters on $\cT_I$. It follows that $\cK \subseteq J$ (because $\cK$ has no characters). Arveson's conjecture is that $\cK = J$. 

The quotient $\cT_I / J$ is commutative, so it equals the algebra of continuous functions on some compact space. What space? (The following two propositions are valid for every configuration of ideals  $I \triangleleft \textrm{Comm}(B) \triangleleft B$ in a C*-algebra $B$).

\begin{proposition}\label{prop:CS}
$\Spec(\cT_I / J) = \pV$, thus $\cT_I / J \cong C(\pV)$. 
\end{proposition}
\begin{proof}
Let $q : \cT_I \rightarrow \cT_I / J$ be the quotient map. We have the adjoint map $q^* : \Spec(\cT_I/J) \rightarrow \Spec(\cT_I) = \pV$, a continuous and injective map, as the adjoint of any quotient map is. We must show that it is surjective. If $\lambda \in \pV$, then $J \subseteq \textrm{ker} \hat{\rho}_\lambda$, so $\hat{\rho}_\lambda$ induces a well defined character $\Theta_\lambda :  \cT_I / J \rightarrow \mb{C}$ given by $\Theta_\lambda(a+J) = \hat{\rho}_\lambda(a)$. Thus $\hat{\rho}_\lambda = q^* (\Theta_\lambda)$, and $q^*$ is surjective. 
\end{proof}

The algebra of interest is $\cO_I$.

\begin{proposition}\label{prop:specOI}
$\Spec(\cO_I) \cong \pV$. 
\end{proposition}

\begin{proof}
We observe that
\bes
\cO_I / (J / \cK) = (\cT_I / \cK) / (J / \cK) \cong \cT_I / J,
\ees
so we identify $\cT_I / J$ as a quotient of $\cO_I$. Consider the following three quotient maps:
\bes
\pi : \cT_I  \rightarrow \cO_I
\ees
\bes
\pi': \cO_I \rightarrow \cT_I / J
\ees
and
\bes
q : \cT_I  \rightarrow \cT_I / J .
\ees
Then $q = \pi' \circ \pi$, therefore $q^* = \pi^* \circ \pi'^*$. By Proposition \ref{prop:CS} $q^*$ is surjective, so $\pi^*$ is surjective, too. Since the adjoint of every quotient map is injective and continuous, $\Spec(\cO_I) \cong \Spec(\cT_I) \cong \pV$.
\end{proof}

\section{Essential norms in $\cA_I$ and the operator algebraic structure of $\pi(\cA_I)$}\label{sec:main}

\subsection{Essential norms and the Gelfand transform}

For $a \in \cA_I$ the {\em essential norm} of $a$ is defined to be the norm of the image of $a$ in $\cO_I$, that is,
\[
\|a\|_e = \|\pi(a)\|, 
\]
where $\pi : \cT_I \rightarrow \cO_I = \cT_I / \cK$ is as in the proof of Proposition \ref{prop:specOI}. Likewise, if $A = \left[a_{ij}\right]_{i,j=1}^n \in M_n(\cA_I)$ is a matrix of elements in $\cA_I$, the essential norm $\|A\|_e$ of $A$ is defined as the norm of its image in $M_n(\cO_I) = M_n(\cT_I) / M_n(\cK)$: 
\[
\|A\|_e = \left\|\left[\pi(a_{ij})\right]_{i,j=1}^n \right\|. 
\]

Given $A = \left[a_{ij}\right]_{i,j=1}^n \in M_n (\cA_I)$, we define its {\em Gelfand transform} $\hat{A} \in M_n(C(\pV))$ by 
\[
\hat{A}(z) = \left[\hat{a}_{i,j}(z) \right]_{i,j=1}^n := \left[\rho_z(a_{i,j}) \right]_{i,j=1}^n  .
\]

\begin{proposition}\label{prop:ess_norm_geq}
For all $A \in M_n(\cA_I)$, 
\be\label{eq:ess_norm_geq}
\|A\|_e \geq \sup_{z \in \pV} \left\|\hat{A}(z)\right\| .
\ee
\end{proposition}
\begin{proof}
In the proof of Proposition \ref{prop:specOI} we have noted the existence of a surjective $*$-representation $\cO_I \rightarrow \cT_I/J = C(\pV)$. 
\end{proof}

\subsection{Banach algebraic spectra}

Let $\cB$ be a commutative and unital Banach algebra, and let $\Spec(\cB)$ be its maximal ideal space. For $b = (b_1, \ldots, b_n)$ a tuple of elements in $\cB$ and $\rho \in \Spec(\cB)$ we denote $\rho(b) = (\rho(b_1), \ldots, \rho(b_n))$. The spectrum of $b$ in $\cB$ is defined to be 
\[
\sigma_\cB(b) = \{\rho(b) : \rho \in \Spec(\cB)\}.
\]
It is an exercise to show that $\sigma_\cB(b)$ is the complement in $\mb{C}^n$ of all $\lambda \in \mathbb{C}^n$ for which there exist $c_1, \ldots, c_n \in \cB$ such that $\sum(b_i - \lambda_i)c_i = 1$. 

If $\cB$ is generated by $b_1, \ldots, b_n$ then there is a natural identification $\sigma_\cB(b) = \Spec(\cB)$, similar to (\ref{eq:spec1}) and (\ref{eq:spec2}).

Let $Z = (Z_1, \ldots, Z_d) = (\pi(S_1), \ldots, \pi(S_d))$ denote the image of $S = (S_1, \ldots, S_d)$ in $\cO_I$, and let $\cB_I:=\overline{\alg}(Z)$ denote the unital norm closed algebra generated by $Z$ in $\cO_I$.

\begin{proposition}\label{prop:specSspecZ}
$\sigma_{\cB_I}(Z) \subseteq \sigma_{\cA_I}(S) = \overline{V}$.
\end{proposition}

\begin{proof}
The identification of $\sigma_{\cA_I}(S)$ with $\overline{V}$ follows from (\ref{eq:spec1}) and (\ref{eq:spec2}). The fact that $\sigma_{\cB_I}(Z)  \subseteq \overline{V}$ follows easily from the fact that $\pi(S_i) = Z_i$, $i=1, \ldots, d$.
\end{proof}

\subsection{The operator algebraic structure of $\pi(\cA_I)$}

Let $\pi(\cA_I)$ denote the image of $\cA_I$ in the quotient $\cO_I$. If Arveson's conjecture is true then $\cO_I \cong C(\partial V)$ and $\overline{\pi(\cA_I)}$ is therefore equal to the norm closed  algebra $A(V)$ generated by the polynomials in $C(\oV)$. The C*-envelope of the latter operator algebra is easily seen to be $C(\pV)$. Thus, Arveson's conjecture -- if true -- would imply the following theorem, which we now prove directly.

\begin{theorem}\label{thm:main}
Let $Z$ denote the image of $S$ in $\cO_I$, and let $\cB_I := \overline{\pi(\cA_I)}$ denote the unital norm closed algebra generated by $Z$. Then
\begin{enumerate}
\item The tuple $Z^*$ is subnormal. 
\item For every $A \in M_n(\cA_I)$, the essential norm of $A$ is equal to $\sup_{z \in \pV}\|\hat{A}(z)\|$. In particular, $\cB_I$ is completely isometrically isomorphic to $A(V)$. 
\item The C*-envelope of $\operatorname{Span} \{I,Z_1,Z_1^*,\ldots,Z_d,Z_d^*\}$ is equal to $C(\partial V)$.
\item The C*-envelope of the algebra $\cB_I$ is equal to $C(\partial V)$. 
\end{enumerate}
\end{theorem}

\begin{proof}
Consider $\cB_I$ as a concrete operator algebra inside $B(H)$. 

Since $Z_1, \ldots, Z_d$ is a commuting row contraction satisfying $\sum Z_i Z_i^* = I_H$ a theorem of Athavale (\cite[Proposition 2]{Athavale}, see also \cite[p. 217]{Arv98}) implies that $Z_1^*, \ldots, Z_d^*$ is subnormal. That proves (1). 

To be precise, the theorem of Athavale says that there exists a tuple $N = (N_1, \ldots, N_1)$ of normal operators on a Hilbert space $K \supseteq H$ such that $N_i^*\big|_H = Z_i^*$, and such that the joint spectrum $X:=\sigma(N)$ of $N$ is contained in $\pBd$. 
We may assume that $N$ is the minimal normal extension of $Z$, and we identify $C^*(N)$ with $C(X)$. 

Let 
\[
N_i = \begin{pmatrix}
  Z_i & 0 \\
  X_i & Y_i\\
\end{pmatrix} .
\]

By Putinar's spectral inclusion theorem \cite{Putinar84}, the Taylor joint spectra satisfies $Sp(N,K)  \subseteq Sp(Z,H)$. Since $N$ is normal, $Sp(N,K) = \sigma(N) = X$, and the inclusion $Sp(Z,H) \subseteq \sigma_{\cB_I}(Z)$ is always true (see Proposition IV.25.3, \cite{MullerBook}). By Proposition \ref{prop:specSspecZ}, $\sigma_{\cB_I}(Z) \subseteq \oV$, thus $X \subseteq \pV$. We will soon see that in fact $X = \pV$. 

If $p$ is a polynomial, then 
\[
p(N) = \begin{pmatrix}
  p(Z) & 0 \\
  * & *\\
\end{pmatrix} ,
\]
whence $\|p(Z)\| \leq \|p(N)\| \leq \sup_{z \in X} |p(z)| \leq \sup_{z \in \pV} |p(z)|$. Using Proposition \ref{prop:ess_norm_geq}, 
\[
\|p(S)\|_e = \|p(Z)\| \leq \sup_{z \in \pV} |p(z)| = \sup_{z \in \pV} |\widehat{p(S)}(z)| \leq \|p(S)\|_e ,
\]
thus $\|p(Z)\| = \sup_{z \in \pV} |p(z)|$. The same argument works for matrix valued polynomials, after reshuffling. Taking norm limits of polynomials, we obtain the first part of (2). The second part of (2) follows immediately.  


The argument that (2) implies (3) and (4) is standard, but for completeness we give some details. 
(For facts about the Choquet boundary of a separating subset of a uniform algebra, we refer the reader to \cite[Section 8]{Phelps}.) 
The subspace $\cZ := \operatorname{Span} \{I,Z_1,\ldots,Z_d\}$, viewed (completely isometrically) as a subspace of $C(\overline{V})$, contains the constant functions and separates points. Hence the closure of the Choquet boundary of $\cZ$ is the Shilov boundary of $\cZ$ in $C(\overline{V})$. By the maximum principle, the Shilov boundary is contained in $\partial V$. On the other hand, for every $\lambda \in \partial V$, the function $f_\lambda (z) = \langle z, \lambda \rangle  = \sum \overline{\lambda_i} z_i$ peaks at $z = \lambda$. Hence the Choquet boundary must be equal to $\partial V$. The C*-envelope of $\cZ$ is therefore $C(\partial V)$. The C*-envelope of $\cZ$ is equal to the C*-envelope of $\cZ + \cZ^* = \operatorname{Span} \{I,Z_1,Z_1^*,\ldots,Z_d,Z_d^*\}$. 

The proof of (4) is similar. 
\end{proof}

\begin{remark}
By \cite[Theorem 1.1]{FangXia11}, there exists a multiplier $f$ on $H^2_d$ such that $\|M_f\|_e > \|f\|_\infty$. Thus the second assertion of the theorem cannot be improved to all multipliers, and is not trivial. 
\end{remark} 

As a consequence of Theorem \ref{thm:main} above we obtain

\begin{corollary}
$\sigma_{\cB_I}(Z) = \oV$. 
\end{corollary}

We also record the following corollary of the proof of Theorem \ref{thm:main}. 

\begin{corollary} \label{cor:c-star-env-comm-tuple}
Let $T = (T_1, \ldots, T_d)$ be a commuting tuple of operators acting on $H$. I
f $\sum T_i T_i^* = I$, then $\overline{\alg}(T)$ is completely isometrically isomorphic to a uniform algebra, and in particular the C*-envelope of $\operatorname{Span}\{I, T_1, T_1^*, \ldots, T_d, T_d^*\}$ is commutative. 
\end{corollary}

\subsection{A C*-envelope approach to Arveson's conjecture}
\label{sec:suggest-strategy}

The results of the previous section suggest the following strategy for proving Arveson's conjecture. Since we already know that the C*-envelope of $\cB_I = \overline{\pi(\cA_I)}$ is isomorphic to $C(\pV)$, it remains to show that the Shilov boundary ideal of $\cB_I$ in $\cO_I$ is trivial. As the Shilov boundary of $\cA_I$ in $\cT_I$ is trivial (Theorem 10.4 in \cite{DRS}), one might hope that this passes to the quotient, i.e., that the Shilov boundary of  $\pi(\cA_I)$ in $\cO_I$ is also trivial. In fact, the research behind this paper was initiated with the hope of following this strategy, and there was some reason to believe it could be carried through because it works for certain nice examples in the commutative case. 

\begin{example}
Let $A$ be a unital subalgebra of $C(X)$, where $X$ is a compact metric space. Let $U$ be an open set in $X$ and let $I$ be the ideal of functions that vanish on $U$. Then if the Shilov boundary ideal of $A$ in $C(X)$ is trivial, and if $\pi : C(X) \rightarrow C(X)/I$ denotes the quotient map, then the Shilov boundary ideal of $\pi(A)$ in $C(X)/I$ is also trivial. Indeed, under these assumptions $I = I_F = \{f \in C(X) : f \big|_F = 0\}$, where $F = \overline{U}$. Then $C(X)/I \cong C(F)$, $\pi$ may be identified with the restriction map $f \mapsto f \big|_F$, and $\pi(A)$ may be identified with the restriction algebra $A_F$.  But by a theorem of Bishop \cite[Corollary 8.2]{Phelps}, the peak points of $A$ correspond to that Choquet boundary of $A$, hence they are dense in $X$. It follows that the peak points of $A_F$ are dense in $F$, so $F$ is the Shilov boundary of $A_F$ in $C(F)$, as required. 
\end{example}

Unfortunately, triviality of the of the Shilov ideal does not always pass to the quotient, even in the commutative case. 
\begin{example}
Let $X = [0,1]$, let $A$ be the algebra of continuous functions vanishing at the point $0$, and let $I_F$ be the ideal of functions vanishing on $F = \{0,1/2\}$. Then the Shilov ideal of $A$ is trivial, while that of $A_F$ is not. 
\end{example}
The following example 
due to Ken Davidson 
shows that in general triviality of the Shilov boundary does not pass to quotients, even in the irreducible case. 

\begin{example}
Let $\{e_n\}$ be an orthonormal basis of a Hilbert space $H$. Let $\{t_n\}_{n=3}^\infty$ be a dense sequence in the unit disc $\mb{D}$ and set
$t_1 = 2$ and $t_2 = 0$. Define $T \in B(H)$ by $T e_n = t_n e_n + \frac{1}{n} e_{n+1}$ for $n=1,2,\ldots$. 

Denote $\cA = \alg(T)$. We first show that $C^*_{e}(\cA) = C^*(\cA)$. Indeed, a computation shows that $T^*T = T^*T \big|_{\operatorname{span}\{e_1\}} \oplus T^*T\big|_{\{e_1\}^\perp}$, and since $T^* T e_1 = 5 e_1$ and $\|T^*T\big|_{\{e_1\}^\perp}\|<5$ it follows that the rank $1$ projection $e_1 e_1^*$ is in $C^*(T)$. One may check that $C^*(T)$ is irreducible, so $\cK \subseteq C^*(T)$. But as the quotient map is not isometric on $\cA$ we have that the Shilov ideal of $\cA$ is trivial and $C^*_{e}(\cA) = C^*(T)$. 

Let $\pi :C^*(T) \rightarrow C^*(T) / \cK$ be the quotient map. We now claim that the triviality of the Shilov ideal does not pass down to $\pi(\cA)$. To see this, note that $T$ is a diagonal-plus-compact operator, so it is essentially normal. The choice of $\{t_n\}$ ensures that $\sigma_e(T) = \overline{\mb{D}}$. It follows that $\overline{\pi(\cA)}$ is completely isometrically isomorphic to the disc algebra $A(\mb{D})$ and that $C^*(T) / \cK \cong C(\overline{\mb{D}})$. Thus $C^*(T)/ \cK$ is not the C*-envelope of $\pi(\cA)$. 
\end{example}

\section{Essential normality and hyperrigidity}
\label{sec:hyperrigidity}

\subsection{Hyperrigidity}

In this section we will establish a connection between the essential normality of a commuting tuple of operators and the behaviour of the tuple as the generating set of a C*-algebra.

A key idea in Arveson's recent work on completely positive maps and the noncommutative Choquet boundary is the following notion of rigidity for a completely positive map (see for example \cite{Arv08} and \cite{Arv11}).

\begin{definition}
Let $\cS$ be an operator system (i.e. a unital self-adjoint subspace) that generates a C*-algebra $\cT$.  A unital completely positive (UCP) map $
\phi:\cS \to \cB(H)$ is said to have the \emph{unique extension property} if it has a unique  extension to a UCP map $\tilde{\phi}:\cT \to \cB(H)$.
\end{definition}

The boundary representations of $\cT$ for $\cS$, which were first singled out in \cite{Arv69}, are precisely the irreducible representations $\pi:\cT \to \cB(H)$ with the property that the restriction $\pi|_\cS$ has the unique extension property. The existence of boundary representations was an open question for some time, but it is now known (see \cite{Arv08} and \cite{DK13}) that boundary representations exist in abundance.

In \cite{Arv11}, Arveson introduced the following definition to highlight the importance of algebras with the property that every representation has the unique extension property.

\begin{definition}
Let $S$ be a generating set of a C*-algebra $\cT$, and let $\cS$ denote the operator system generated by the elements in $S$. The set $S$ is said to be \emph{hyperrigid} if for every nondegenerate representation $\pi:\cT \to \cB(H)$, the restriction $\pi|_\cS$ has the unique extension property, i.e. $\pi$ is the unique UCP extension of $\pi |_\cS$ to $\cT$.
\end{definition}

\begin{remark}
We note that this is not Arveson's original definition of hyperrigidity. However, by \cite[Theorem 2.1]{Arv11}, it is completely equivalent, and we feel the definition given here fits better with the theme of our paper.
\end{remark}

An important property of hyperrigidity is that it passes to quotients \cite[Corollary 2.2]{Arv11}. This will be crucial for what follows, and should be compared to the remarks in Section \ref{sec:suggest-strategy}.

The main result in this section, Theorem \ref{thm:main-hr}, is that the $d$-tuple $S = (S_1,\ldots,S_d)$ defined in the introduction is essentially normal if and only if it is hyperrigid.

\subsection{Facts about singular representations}

We continue to use the notation from the introduction, and in addition define the operator system $\cS = \operatorname{Span} \{I,S_1,S_1^*,\ldots,S_d,S_d^* \}$.
A representation $\rho:\mathcal{T}_{I}\to\mathcal{B}(H)$ is said to be \emph{singular} if it annihilates the compact operators $\cK$.
In this section, we will gather some facts about singular representations of $\cT_I$.

\begin{lemma}
\label{lem:top-right-corner}
Let $\rho:\mathcal{T}_{I}\to\mathcal{B}(H)$
be a singular representation, and let $\pi:\cT_I \to \cB(K)$ be a  representation such that $\pi |_{\cS}$ is a dilation of  $\rho |_{\cS}$.
Then the subspace $H$ is coinvariant for $\pi(\cA_I)$. 
\end{lemma}

\begin{proof}
With respect to the decomposition $K=H\oplus H^{\perp}$
we can write
\[
\pi(S_{i})=\left(\begin{array}{cc}
\rho(S_{i}) & X_{i}\\
Y_{i} & Z_{i}
\end{array}\right),\quad1\leq i\leq d.
\]
Note that $X_{i}=P_{H}\pi(S_{i})|_{H^{\perp}}$.
We must prove that $X_{i}=0$ for each $i$.

We can write the rank one projection $P_\mathbb{C}$ onto $\mathbb{C} = E_0$ as
\[
P_\mathbb{C} = I - \sum_{i = 1}^d M_{z_i} M_{z_i}^*,
\]
which implies
\[
\sum_{i=1}^{d}S_{i}S_{i}^{*}=I-P_{\cF_I}P_{\mathbb{C}} |_{\cF_I}.
\]
Then since $\rho$ annihilates $\cK$, which, in particular, contains $P_{\cF_I}P_{\mathbb{C}} |_{\cF_I}$,
\[
\sum_{i=1}^{d}\rho(S_{i}S_{i}^{*})=I.
\]
This gives
\begin{eqnarray*}
I & \geq & \sum_{i=1}^{d}P_{H}\pi(S_{i}S_{i}^{*})|_{H}\\
 & = & \sum_{i=1}^{d}\rho(S_{i}S_{i}^{*})+\sum_{i=1}^{d}X_{i}X_{i}^{*}\\
 & = & I+\sum_{i=1}^{d}X_{i}X_{i}^{*}.
\end{eqnarray*}
Therefore, $\sum_{i=1}^{d}X_{i}X_{i}^{*}=0$, and hence $X_{i}=0$
for each $i$.
\end{proof}

\begin{lemma}\label{lem:agree}
\label{lem:special-products}
Let $\rho:\mathcal{T}_{I}\to\mathcal{B}(H)$
be a singular representation, and let $\phi:\cT_I \to \cB(H)$ be a UCP map such that $\phi |_{\cS} = \rho |_{\cS}$.
Then $\phi(T) = \rho(T)$ for every $T$ in the closure of $\operatorname{span}\{A_{1}A_{2}^{*}\mid A_{1},A_{2}\in\mathcal{A}_{I}\}$,
\end{lemma}

\begin{proof}
Let $\pi:\cT_I \to \cB(K)$ be a Stinespring dilation of $\phi$, so that
\[
\phi(T) = P_H \pi(T) |_H, \quad T \in \cT_I.
\]
Then in particular, $\pi |_{\cS}$ is a dilation of $\rho |_{\cS}$. Hence by Lemma \ref{lem:top-right-corner}, the subspace $H$ is coinvariant for $\pi(\cA_I)$.
It follows immediately that for $A_1,A_2 \in \cA_I$,
\[
\phi(A_1 A_2^*) = P_H \pi(A_1 A_2^*) |_H = P_H \pi(A_1) P_H \pi(A_2^*) |_H = \rho(A_1 A_2^*),
\]
and the result follows by continuity after taking linear combinations.
\end{proof}

\begin{lemma}\label{lem:ann_compact}
\label{lem:kills-compacts}Let $\rho:\mathcal{T}_{I}\to\mathcal{B}(H)$
be a singular representation, and let $\phi:\mathcal{T}_{I}\to\mathcal{B}(H)$
be a UCP map such that $\phi|_{\cS}=\rho|_{\cS}.$ Then $\phi$ annihilates the compact operators.
\end{lemma}

\begin{proof}
First, we require the fact that the closure of the set
\[
\cA_I \cA_I^* = \operatorname{span}\{A_{1}A_{2}^{*}\mid A_{1},A_{2}\in\mathcal{A}_I\}
\]
contains the compact operators (this fact is appears in \cite[Theorem 1.3]{Popescu06} under some additional assumptions, but since we require this fact in slightly greater generality, we provide a proof). 
To see this, write the rank one projection $P_\mathbb{C}$ onto $\mathbb{C} = E_0$ as
\[
P_\mathbb{C} = I - \sum_{i = 1}^d M_{z_i} M_{z_i}^*,
\]
and observe that, in particular, $P_\mathbb{C}$ is contained in $\cA_d \cA_d^*$. Since the subspace $\mathbb{C}$ is cyclic for $\cA_d$, it follows that $\cA_d \cA_d^*$ contains enough rank one operators. Hence
$\cA_I \cA_I^* = P_{\cF_I} \overline{\cA_d \cA_d^*} |_{\cF_I}$ contains every rank one operator on $\cF_I$, and it follows that the closure of this set contains the compacts.

Now, since  $\phi |_{\cS} = \rho |_{\cS}$, Lemma \ref{lem:agree} implies that $\phi(T) = \rho(T)$ for every $T$ in the closure of $\cA_I \cA_I^*$. From above, this set contains the compact operators. Since $\rho$ is singular and annihilates the compact operators, it follows that $\phi$ does as well.
\end{proof}

\begin{lemma}
\label{lem:dilation-singular}Let $\rho:\mathcal{T}_{I}\to\mathcal{B}(H)$
be a singular representation, and let $\pi:\mathcal{T}_{I}\to\mathcal{B}(K)$
be a representation such that $\pi|_{\cS}$ is a dilation
of $\rho|_{\cS}$. Then there is a subrepresentation $\pi_s:\mathcal{T}_{I}\to\mathcal{B}(K_s)$
of $\pi$ such that $\pi_s$ is singular and
$\pi_s|_{\cS}$ is a dilation of $\rho|_{\cS}$.\end{lemma}
\begin{proof}
By standard facts about representations of C*-algebras, we can decompose $\pi=\mathrm{id}^{(\alpha)}\oplus\pi_s$, where $\mathrm{id}^{(\alpha)}$
is a multiple of the identity representation and $\pi_s:\mathcal{T}_{I}\to\mathcal{B}(K_s)$ is a singular representation.
The subspace
$K_s$ can be written as
\[
(K_s)^{\perp} = \vee \operatorname{Ran} \pi(Q),
\]
where the join is taken over all finite rank projections $Q$. By Lemma \ref{lem:kills-compacts}, $P_H \pi(Q) |_H = 0$ for every such $Q$. Thus $P_{H}(K_s)^{\perp}=0$, which implies
$H\subseteq K_s$, and hence that $\pi_s|_{\cS}$ is a dilation of $\rho|_{\cS}$, as required.
\end{proof}

\begin{proposition}
\label{prop:maximal-singular-dilation}Let $\rho:\mathcal{T}_{I}\to\mathcal{B}(H)$
be a singular representation. Then there is a singular representation $\pi:\mathcal{T}_{I}\to\mathcal{B}(K)$
such that $\pi|_{\cS}$ is a dilation of $\rho|_{\cS}$, and such that $\pi|_{\cS}$ has the unique extension property.
Moreover, the image
$\pi(\mathcal{T}_{I})$ is commutative.
\end{proposition}
\begin{proof}
The existence of a representation $\pi:\mathcal{T}_{I}\to\mathcal{B}(K)$ such that $\pi|_{\cS}$ is a dilation of $\rho|_{\cS}$, and such that $\pi|_{\cS}$ has the unique extension property is implied by the results in \cite{Arv08} (see also \cite{DK13}). It is easy to see that if $\pi$ has the unique extension property when restricted to $\cS$, then every subrepresentation also has this property. Hence by Lemma \ref{lem:dilation-singular},
we can suppose that $\pi$ is singular.

Since $\pi$ annihilates the compacts, we can factor $\pi = \pi_0 \circ q$, where $\pi_0:\cO_I \to \cB(K)$ is a $*$-representation and $q:\cB(\cF_I) \to \cB(\cF_I)/\cK$ denotes the quotient map onto the Calkin algebra. 
Let $\cZ = \cS / \cK = \operatorname{Span}\{ I, Z_1, Z_1^*, \ldots, Z_d, Z_d^*\}$. 
We claim that $\pi_0 |_\cZ$ also has the unique extension property. To see this, let $\phi:\cO_I \to \cB(K)$ be a UCP map such that $\phi |_\cZ = \pi_0 |_\cZ$. 
Then $(\phi \circ q) |_\cS = \pi |_\cS$. Since $\pi |_\cS$ has the unique extension property, it follows that $\phi \circ q = \pi = \pi_0 \circ q$, and hence that $\phi = \pi_0$.

Dritschel and McCullough \cite{DM05} showed that since $\pi_0 |_\cZ$ has the unique extension property, the unique UCP extension $\pi_0$ is a $*$-representation that factors through the C*-envelope of $\cZ$, which by Theorem \ref{thm:main} is $C(\partial V)$. 
In particular, the image $\pi_0(\cO_I) = \pi(\cT_I)$ is commutative. 
\end{proof}

\subsection{Hyperrigidity and essential normality}

\begin{lemma}\label{lem:ENiffCPT}
The $d$-tuple $S$ is essentially normal if and only if $I - \sum_{i=1}^d S_i^* S_i$ is compact. 
\end{lemma}
\begin{proof}
We have already noted that $I - \sum S_i S_i^* \in \cK$. Thus if $S$ is essentially normal then $I-\sum S_i^*S_i$ is equal to $I- \sum S_i S_i^*$ mod compacts, so is compact. 

Conversely, if $I - \sum S^*_i S_i \in \cK$ then $\sum_{i=1}^d [S_i^*, S_i] \in \cK$. 
Using the essential normality of the $d$-shift $M_z$, one checks that $\pi([S_i,S_i^*]) \geq 0$ for all $i$. It follows that $\pi([S_i,S_i^*]) = 0$ for all $i$. From Fuglede's theorem is it follows that $\pi([S_i,S_j^*]) = 0$ for all $i,j$, as required.  
\end{proof}

\begin{proposition}
\label{prop:sum-to-1-ess-norm} Suppose that $S$ is essentially normal, and let $\rho:\mathcal{T}_{I}\to\mathcal{B}(H)$
be a singular representation. 
Then the restriction $\rho|_{\cS}$ has the unique extension property.
\end{proposition}

\begin{proof}
We will use the fact that a UCP  map $\theta$ has the unique extension property if and only if it is {\em maximal}, meaning that every UCP map that dilates $\theta$ contains it as a direct summand \cite[Proposition 2.4]{Arv08}. 
Let $K$ be a Hilbert space properly containing $H$,
and let $\pi:\mathcal{T}_{I}\to\mathcal{B}(K)$ be a representation
such that the restriction $\pi|_{\cS}$ is a dilation of $\rho|_{\cS}$.
We need to show that this dilation is trivial, i.e. that $\pi|_{\cS}=\rho|_{\cS}\oplus\phi$
for some UCP map $\phi$. 

By Lemma \ref{lem:top-right-corner}, we
can decompose $K=H\oplus H^{\perp}$
and write
\[
\pi(S_{i})=\left(\begin{array}{cc}
\rho(S_{i}) & 0\\
Y_{i} & Z_{i}
\end{array}\right),\quad1\leq i\leq d.
\]
Showing that $\pi|_{\cS}$ is a trivial dilation of $\rho|_{\cS}$ is equivalent
to showing that $Y_{i}=0$ for each $i$.

As in the proof of Lemma \ref{prop:maximal-singular-dilation}, we can apply Lemma \ref{lem:dilation-singular} and suppose that $\pi$
is singular. 
By Lemma \ref{lem:ENiffCPT}, $I - \sum S_i^* S_i \in \cK$. From singularity, 
\[
\sum_{i=1}^{d}\pi(S_{i}^{*}S_{i}) = I_K \,\,\,\, , \,\,\,\, \sum_{i=1}^{d}\rho(S_{i}^{*}S_{i}) = I_H \, ,
\]
which implies
\begin{eqnarray*}
I & = & \sum_{i=1}^{d}P_{H}\pi(S_{i}^{*}S_{i})|_{H}\\
 & = & \sum_{i=1}^{d}(\rho(S_{i}^{*}S_{i})+Y_{i}^{*}Y_{i})\\
 & = & I+\sum_{i=1}^{d}Y_{i}^{*}Y_{i}.
\end{eqnarray*}
Therefore, $\sum_{i=1}^{d}Y_{i}^{*}Y_{i}=0$, and hence $Y_{i}=0$
for each $i$.
\end{proof}

\begin{proposition}
\label{prop:irr-rep-bdy-rep}
Suppose that $S$ is essentially normal. 
Then every irreducible representation of $\cT_I$ is a boundary representation for $\cS$. 
\end{proposition}

\begin{proof}
By \cite[Proposition 6.4.6]{ChenGuo}, the restriction of the identity representation of $\cT_I$ to $\cS$ is irreducible, and has the unique extension property. 
Since every $*$-representation of $\cT_I$ splits
as the direct sum of a multiple of the identity representation and a singular representation, it follows that the identity representation is the only irreducible non-singular representation of $\cT_I$. 
By Proposition \ref{prop:sum-to-1-ess-norm}, the restriction of every irreducible singular representation of $\cT_I$ to $\cS$ also has the unique extension property.
\end{proof}

\begin{theorem}
\label{thm:main-hr}
The $d$-tuple $S=(S_{1},\ldots,S_{d})$ is essentially normal if and only if it is hyperrigid.
\end{theorem}
\begin{proof}
If $S$ is hyperrigid, then since hyperrigidity passes to quotients, the image $Z$ of $S$ modulo the compacts is also hyperrigid. Hence the C*-envelope of $Z$ is $\mathrm{C}^*(Z) = \cO_I$ \cite[Corollary 4.2]{Arv11}. But we know by Theorem \ref{thm:main} that the C*-envelope of $Z$ is commutative, meaning $\cO_I$ is commutative, and hence that $S$ is essentially normal.

Conversely, if $S$ is essentially normal, then Proposition \ref{prop:sum-to-1-ess-norm} implies that the restriction
$\pi|_{\cS}$ has the unique extension property for every singular $*$-representation $\pi$. 
By \cite[Proposition 6.4.6]{ChenGuo}, the restriction of the identity representation to $\cS$ has the unique extension property. 
Since every $*$-representation of $\mathcal{T}_{I}$
splits as the direct sum of a multiple of the identity representation and a singular representation, and since the unique extension property passes to direct sums,
it follows that the restriction of every $*$-representation of $\mathcal{T}_{I}$ to $\cS$ has the unique extension property.
\end{proof}

A conjecture of Arveson \cite[Conjecture 4.3]{Arv11}, if true, would imply that $\cS$ is hyperrigid, and hence essentially normal, if and only if every irreducible representation of $\cT_I$ is a boundary representation for $\cS$. This would imply the converse of Proposition \ref{prop:irr-rep-bdy-rep}.

\subsection{The obstruction}
If $S$ is not essentially normal, then by Theorem \ref{thm:main-hr}, it is not hyperrigid. 
Hence there is a *-representation of $\cT_I$ that does not
have the unique extension property when restricted to $\cS$. 
In this section, we identify how such a map can arise.

Let $\rho:\mathcal{T}_{I}\to\mathcal{B}(H)$ be
a singular representation, and let $\pi:\mathcal{T}_{I}\to\mathcal{B}(K)$
be a representation as in Proposition \ref{prop:maximal-singular-dilation}, i.e.
a singular representation such that $\pi|_{\cS}$ is a dilation of $\rho|_{\cS}$, and such that $\pi|_{\cS}$ has the unique extension property.
Then in particular, the image $\pi(\mathcal{T}_{I})$
is commutative. By Lemma \ref{lem:top-right-corner}, we can decompose
$K=H\oplus H^{\perp}$ and write
\[
\pi(S_{i})=\left(\begin{array}{cc}
\rho(S_{i}) & 0\\
Y_{i} & Z_{i}
\end{array}\right),\quad1\leq i\leq d.
\]
Since the image $\pi(\mathcal{T}_{I})$ is commutative,
in particular we have
\[
\pi(S_{i}^{*}S_{j})=\pi(S_{j}S_{i}^{*}),\quad1\leq i,j\leq d,
\]
which implies
\[
\rho(S_{j}S_{i}^{*}) = \rho(S_{i}^{*}S_{j}) + Y_{i}^{*}Y_{j},\quad1\leq i,j\leq d,.
\]
Let $\phi:\mathcal{T}_{I}\to\mathcal{B}(H)$ denote the
UCP map defined by 
\[
\phi(A)=P_{H}\pi(A)|_{H},\quad\forall A\in\mathcal{T}_{I}.
\]
Then from above, $\phi|_{\cS}=\rho|_{\cS}$,
but
\[
\phi(S_{i}^{*}S_{j})=\rho(S_{i}^{*}S_{j})+Y_{i}^{*}Y_{j}=\rho(S_{j}S_{i}^{*}),\quad1\leq i,j\leq d.
\]
Thus we obtain the following result.

\begin{proposition}
\label{prop:strange-map}
Let $\rho:\mathcal{T}_{I}\to\mathcal{B}(H)$ be a singular representation. Then there is a UCP map $\phi:\cT_I \to \cB(H)$ such that
$\phi|_{\cS}=\rho|_{\cS}$, but
\[
\phi(S_{i}^{*}S_{j})=\rho(S_{j}S_{i}^{*}),\quad1\leq i,j\leq d.
\]
\end{proposition}

One consequence of Proposition \ref{prop:strange-map} is the following equivalence between essential normality and the approximability of certain products of elements in $\cS$.

\begin{proposition}
The $d$-tuple $S = (S_1,\ldots,S_d)$ is essentially normal if and only if $S_i^*S_j$ belongs to the closure of $\operatorname{Span}\{A_{1}A_{2}^{*}\mid A_{1},A_{2}\in\mathcal{A}_{I}\}$ for all $1 \leq i,j \leq d.$
\end{proposition}

\begin{proof}
Suppose $S_i^*S_j$ belongs to the closure of $\operatorname{span}\{A_{1}A_{2}^{*}\mid A_{1},A_{2}\in\mathcal{A}_{I}\}$ for all $1 \leq i,j \leq d.$
Let $q:\cB(\cF_I) \to \cB(\cF_I)/\cK$ denote the quotient map onto the Calkin algebra, and let $\phi$ be as in Proposition \ref{prop:strange-map}. Then by Lemma
\ref{lem:special-products} and Proposition \ref{prop:strange-map},
\[
q(S_{j}S_{i}^{*})=\phi(S_{i}^{*}S_{j})=q(S_{i}^{*}S_{j}),\quad\forall1\leq i\leq d.
\]
Hence $S_{i}^{*}S_{j}=S_{j}S_{i}^*+K$ for some compact operator $K \in \cK$ for all $i,j$, so $S$ is essentially normal. Since $\cK$ is contained in $\operatorname{span}\{A_{1}A_{2}^{*}\mid A_{1},A_{2}\in\mathcal{A}_{I}\}$ (see the proof of Lemma \ref{lem:ann_compact}), the proof of the converse is immediate.
\end{proof}

\section{Other spaces}
\label{sec:other-spaces}

\subsection{For what modules do the results hold?}
In this paper we have concentrated our attention on operator algebras arising from a certain kind of polynomial ideals, where our starting point was the Hilbert space $H^2_d$ and the row contraction $M_z = (M_{z_1}, \ldots, M_{z_d})$. The universality of this space justifies this seemingly narrow point of view. However, it is interesting to know that these results hold in a wider setting. 

Let $X$ be a compact subset of the closed unit ball $\overline{\mb{B}_d}$. Let $H$ be a Hilbert space and let $T = (T_1, \ldots, T_d)$ be a commuting row contraction. Let $C^*(T)$ and $\overline{\alg}(T)$ denote the unital C*-algebra and the unital norm closed algebra, respectively, generated by $T$. Denote by $\cK$ the compact operators on $H$. Assume the following conditions:
\begin{enumerate}
\item[(1)] $I - \sum_{i=1}^d T_i T_i^* \in \cK$, 
\item[(2)] $\cK \subset C^*(T)$, 
\item[(3)] $\Spec(\overline{\alg}(T)) \cong X$ via the map $\rho \leftrightarrow \rho(T)$.
\end{enumerate} 

Under these assumptions we have that $\Spec(C^*(T)) = X \cap \pBd$, by applying the same proof as that of Proposition \ref{prop:spectra} (here we will need the elementary fact that if $U = (u_{ij})_{i,j=1}^d$ is a unitary matrix then $\tilde{T}_i = \sum_{j=1}^d u_{ij}T_j$ is also a commuting row contraction generating $C^*(T)$). The rest of the results in Section \ref{sec:C*spectra} also follow, and in particular $\Spec(C^*(T)/ \cK) = X \cap \pBd$. 

The results leading up to Theorem \ref{thm:main}, as well as the proof of that theorem also hold in this more general setting. In particular, we obtain that the essential norm of every $A \in M_n(\overline{\alg}(T))$ is given by $\sup_{z \in X \cap \pBd}\|\hat{A}(z)\|_{M_n}$, that $\pi(\overline{\alg}(T))$ is completely isometrically isomorphic to the norm closure of the polynomials on $X \cap \pBd$, and thus its C*-envelope is isomorphic to $C(X \cap \pBd)$. 

If, in addition, we assume the following:
\begin{enumerate}
\item[(4)] $I - \sum_{i=1}^d T_i^* T_i \in \cK$ if and only if $T$ is essentially normal,
\item[(5)] $\cK \subset  \overline{\operatorname{Span}} \{A_1 A_2^* \mid A_1,A_2 \in \operatorname{Alg}(T) \}$,
\item[(6)] the restriction of the identity representation on $\overline{\alg}(T)$ to \linebreak $\operatorname{Span}\{I,T_1,T_1^*,\ldots,T_d,T_d^*\}$ has the unique extension property,
\end{enumerate}
then the results in Section \ref{sec:hyperrigidity} also hold for $T$, using nearly identical proofs. In particular, $T$ is essentially normal if and only if it is hyperrigid.

\subsection{Certain quotients of Besov-Sobolev spaces} 
For $\sigma > 0$, let $B^2_\sigma = B^2_\sigma(\Bd)$ denote the reproducing kernel Hilbert space on $\Bd$ with kernel 
\[
k^\sigma (z,w) = \frac{1}{(1 - \lip z, w \rip)^{2 \sigma}}.
\]
This scale of spaces includes the Drury-Arveson space $H^2_d$ ($\sigma = 1/2$), the Hardy space on the ball ($\sigma = d/2$) and the Bergman space on the ball ($\sigma = (d+1)/2$). It is well known (and easy to check) that this Hilbert space is the completion of the polynomials with respect to the inner product in which all monomials are orthogonal, and for which
\[
\|f\|_{B^2_\sigma}^2 = c_{\sigma,n} \|f\|^2_{H^2_d}, 
\]
for every homogeneous polynomial of degree $n$, where 
\[
c_{\sigma,n} = \frac{\Gamma(n+1)\Gamma(2\sigma)}{\Gamma(2\sigma + n)}.
\]
Below we shall denote $c_n = c_{\sigma,n}$ when $\sigma$ is held fixed. Direct computations show that if $f$ is a homogeneous polynomial of degree $n \geq 1$, then 
\be\label{eq:Besov_row_contraction}
\sum_{i=1}^d M_{z_i} M_{z_i}^* f = \frac{c_{\sigma,n}}{c_{\sigma,n-1}} f = \frac{n}{n+2\sigma-1} f,
\ee
and similarly that
\be
\label{eq:Besov_row_contraction-rev}
\sum_{i=1}^d M_{z_i}^* M_{z_i} f = \frac{n+d}{n+2\sigma}f.
\ee
Thus for $\sigma \geq 1/2$, $M_z = (M_{z_1}, \ldots, M_{z_d})$ is a row contraction satisfying
\be\label{eq:sumI}
I - \sum_{i=1}^d M_{z_i} M_{z_i}^* \in \cK
\ee 
and
\be \label{eq:sum-mk}
I - \sum_{i=1}^d M_{z_i}^* M_{z_i} \in \cK.
\ee
Moreover, a computation shows that $M_z$ is an essentially normal tuple. 

Now let us fix $\sigma$ satisfying $\sigma \geq 1/2$. Let $I$ be a homogeneous ideal in $\mb{C}[z]$ of infinite co-dimension, define $H = B^2_\sigma(\mb{B}_d) \ominus I$, and denote by $T$ the compression of $M_{z}$ to $H$. Note that in this situation, $T_i^* = M_{z_i}^*\big|_{H}$ for all $i$. As before, denote by $\oV = \overline{V(I)} = \cZ(I) \cap \overline{\Bd}$, and as usual, let $\cK$ denote the compacts on $H$. 

\begin{proposition}\label{prop:Besov_Sobolev}
The tuple $T$ is a commuting row contraction, which satisfies the following:
\begin{enumerate}
\item $I - \sum_{i=1}^d T_i T_i^* \in \cK$ and $\cK \subseteq C^*(T)$,
\item $\Spec(\overline{\alg}(T)) \cong \oV$ via the map $\rho \leftrightarrow \rho(T)$,
\item $I - \sum_{i=1}^d T_i^* T_i \in \cK$ if and only if $T$ is essentially normal.
\end{enumerate} 
If $\sigma <d/2$, then $T$ also satisfies
\begin{itemize}
\item[(4)] The restriction of the identity representation on $\overline{\alg}(T)$ to \allowbreak $\operatorname{Span}\{I,T_1,T_1^*,\ldots,T_d,T_d^*\}$ has the unique extension property.
\end{itemize}
\end{proposition}

\begin{proof}
The tuple $T^*$ is the restriction of a commuting column contraction, so $T$ is a commuting row contraction. By (\ref{eq:sumI}) $I_H - \sum T_i T_i^* = P_H(I_{B^2_\sigma} - \sum M_{z_i}M_{z_i}^*)P_H \in \cK$. To see $\cK \subseteq C^*(T)$ one can argue as in \cite[Proposition 2.5]{Arv07}. This gives (1).

To see (2), observe that the map $\rho \leftrightarrow \rho(T)$ gives rise to a homeomorphism between $\Spec(\overline{\alg}(T))$ and a compact subset $X \subset \mb{C}^d$. We will show that $X = \oV$. 

We begin by observing that $H = \mb{C} \oplus H_1 \oplus H_2 \oplus \ldots$, where $H_n$ is the complement of the space of $n$th degree homogeneous polynomials in $I$ inside the space of $n$th degree homogeneous polynomials in $B^2_\sigma$. Now if $p \in I$ and $h \in H_n$, then 
\[
p(T)h = P_H p(M_z) P_H h = P_H ph = 0,
\]
because $ph \in I \perp H$. Thus $p(T) = 0$. If $\rho \in \Spec(\overline{\alg(T)})$ is associated with $\rho(T) = \lambda \in X$, then $p(\lambda) = p(\rho(T)) = \rho(p(T)) = 0$. This shows that $\lambda \in \cZ(I)$. In addition, since every character of an operator algebra is completely contractive, $|\lambda| \leq 1$. Thus $X \subseteq \oV$. 

To show that $\oV \subset X$, it suffices to show that 
\be\label{eq:point_inequality}
|p(\lambda)| \leq \|p(T)\|
\ee
for all $\lambda \in \oV$ and every polynomial $p$, because then $p(T) \mapsto p(\lambda)$ is the character associated to the point $\lambda$. Clearly, if (\ref{eq:point_inequality}) holds for all $\lambda \in V$ it holds also for all $\lambda \in \oV$. To prove (\ref{eq:point_inequality}) for $\lambda \in V$ we will explicitly exhibit the evaluation functional as a vector state. For this we use an idea that has already been used by various authors (see \cite{Arv98, MullerVasilescu,Popescu99}). It is hard to find a convenient reference that precisely fits our needs, so we give the details. 

Given $\lambda \in V$, we define 
$v_\lambda \in B^2_\sigma$ by 
\[
v_\lambda = C_\lambda\sum_{n=0}^\infty c_n^{-1} \sum_{|\alpha|=n} \frac{|\alpha|!}{\alpha!} \overline{\lambda}^\alpha z^\alpha ,
\]
where $C_\lambda$ is a normalization constant to be determined later. 
We compute 
\begin{align*}
\|v_\lambda\|^2_{B^2_\sigma} &= C_\lambda^2 \sum_{n=0} c_n^{-2} \sum_{|\alpha|=n} \left(\frac{|\alpha|!}{\alpha!}\right)^{2} |\lambda|^{2\alpha} \| z^\alpha \|^2_{B^2_\sigma} \\
&= C_\lambda^2 \sum_{n=0} c_n^{-2} \sum_{|\alpha|=n} \left(\frac{|\alpha|!}{\alpha!}\right)^{2} |\lambda|^{2\alpha} c_n \| z^\alpha \|^2_{H^2_d} \\
&= C_\lambda^2 \sum_{n=0}^\infty c_n^{-1}\sum_\alpha \frac{|\alpha|!}{\alpha!} |\lambda|^{2\alpha} \\
&= C_\lambda^2 \sum_{n=0}^\infty c_n^{-1} (|\lambda_1|^2 + \ldots + |\lambda_d|^2)^n.
\end{align*} 
Since $c_n \sim n^{2\sigma -1}$ and $|\lambda|^2 = |\lambda_1|^2 + \ldots + |\lambda_d|^2 < 1$, the right hand side converges, and we may choose $C^2_\lambda$ so that it converges to $1$. 
This shows in particular that the sum defining $v_\lambda$ converges in norm. 

We claim that $v_\lambda \in H$. Indeed, if $p = \sum_{|\beta|=n} a_\beta z^\beta \in B^2_\sigma$ is homogeneous of degree $n$, then 
\bes
\lip p, v_\lambda \rip = C_\lambda c_n^{-1} \sum_{|\alpha|=n} a_\alpha \lambda^\alpha \frac{|\alpha|!}{\alpha!} \|z^\alpha\|^2_{B^2_\sigma} = C_\lambda  p(\lambda). 
\ees
From linearity, $\lip f, v_\lambda \rip = C_\lambda f(\lambda)$ for all $f \in B^2_\sigma$ (thus $v_\lambda$ is the normalized kernel function of $B^2_\sigma$, and in particular). In particular if $p \in I$, then $\lip p, v_\lambda \rip = 0$ , so $v_\lambda \in B^2_\sigma \ominus I = H$. 

We now define a functional $\rho : \alg(T) \rightarrow \mb{C}$ by 
\[
\rho(p(T)) = \lip p(T) v_\lambda, v_\lambda \rip . 
\]
Since $\|v_\lambda\| = 1$, $\|\rho\| = 1$ and it extends to a functional on $\overline{\alg(T)}$. But for every polynomial $p$
\[
\rho(p(T)) = \lip p(M_z) v_\lambda, v_\lambda \rip = \lip pv_\lambda, v_\lambda \rip = C_\lambda p(\lambda) v_\lambda(\lambda) =  p(\lambda). 
\]

Finally, an easy computation shows that the tuple $M_z$ is essentially normal. Proceeding as in the proof of \ref{lem:ENiffCPT} and applying (\ref{eq:sum-mk}) implies (3). Assertion (4) follows as in the proof of \cite[Proposition 6.4.6]{ChenGuo} (and note that it fails for $\sigma \geq d/2$).
\end{proof}


\begin{lemma}
\label{lem:kills-compacts-2}For $\sigma > 1/2$, let $\rho:C^*(T) \to\mathcal{B}(H)$
be a singular representation, and let $\phi:C^*(T) \to\mathcal{B}(H)$
be a UCP map such that $\phi|_{\cS}=\rho|_{\cS}.$ Then $\phi$ annihilates the compact operators.
\end{lemma}

\begin{proof}
We will first show that for every $n \geq 0$, $\phi(P_{H_n}) = 0$.
For a homogeneous polynomial $f$ of degree $n \geq 0$, we compute
\[
\left(I - \sum_{i = 1}^d M_{z_i} M_{z_i}^*\right)f = \left(1 - \frac{n}{n+2\sigma-1}\right)f = \frac{2\sigma-1}{n+2\sigma-1}f.
\]
In particular, $I - \sum_{i = 1}^d M_{z_i} M_{z_i}^* \in \cK$, and hence 
\[
I - \sum_{i = 1}^d T_i T_i^* = P_H \left. \left(I - \sum_{i = 1}^d M_{z_i} M_{z_i}^* \right) \right|_H \in \cK.
\]
Thus, since $\rho$ is singular, $\rho (I - \sum_{i = 1}^d T_i T_i^*) = 0$.

As in the proof of Lemma \ref{lem:agree}, $\phi(T) = \rho(T)$ for every $T$ in the closure of $\operatorname{span}\{A_1A_2^* \mid A_1,A_2 \in \cA_I\}$. Hence $\phi (I - \sum_{i = 1}^d T_i T_i^*) = 0$. 

From above, for every $n \geq 0$, there is a constant $C > 0$ such that $P_{H_n} \leq C (I - \sum_{i = 1}^d T_i T_i^*)$. Thus
\[
\phi(P_{H_n}) \leq C \phi \left(I - \sum_{i = 1}^d T_i T_i^*\right) = 0,
\]
and hence $\phi(P_{H_n}) = 0$.

Let $\pi : C^*(T) \to \cB(K)$ be a minimal Stinespring dilation of $\phi$, so that
\[
\phi(T) = P_H \pi(T) |_H, \quad T \in C^*(T) . 
\]
We can decompose $\pi = \operatorname{id}^{(\alpha)} \oplus \pi_s$, where $\operatorname{id}^{(\alpha)}$ is a multiple of the identity representation and $\pi_s : C^*(T)  \to \cB(K_s)$ is a singular representation. Then we have
\[
(K_s)^\perp = \vee_n \operatorname{Ran} \pi(P_{H_n}).
\]
Since $\phi(P_{H_n}) = 0$ for each $n \geq 0$, it follows that $(K_s)^\perp \subset H^\perp$, i.e. $H \subset K_s$. By minimality, it follows that $\pi = \pi_s$, and hence that $\phi$ annihilates the compact operators.

\end{proof}

From the discussion at the beginning of the section and the last proposition we obtain the following corollary, where Lemma \ref{lem:kills-compacts-2} is used in place of Lemma \ref{lem:kills-compacts} for $\sigma > 1/2$.

\begin{corollary}\label{cor:Besov_Sobolev}
Let $\sigma \geq 1/2$. Let $I$ be a homogeneous ideal of infinite co-dimension in $\mb{C}[z]$, and let $T$ be the compression of $M_z$ to $H = B^2_\sigma \ominus I$. Let $R$ denote the image of $T$ in $C^*(T)/\cK$, and write $V = V(I)$ and $\partial V = \overline{V} \cap \partial\mb{B}_d$. Then
\begin{enumerate}
\item The tuple $R^*$ is subnormal, 
\item For every $A \in M_n(\alg(T))$, the essential norm of $A$ is equal to $\sup_{z \in \pV}\|\hat{A}(z)\|$. In particular, $\overline{\alg}(R)$ is completely isometrically isomorphic to $A(V)$,
\item The C*-envelope of $\overline{\alg}(R)$ is equal to $C(\partial V )$.
\end{enumerate}
If $\sigma < d/2$ then we also have
\begin{itemize}
\item[(4)] The tuple $T$ is essentially normal if and only if it is hyperrigid.
\end{itemize}
\end{corollary}
\begin{remark}
Note that Corollary \ref{cor:Besov_Sobolev} {\em would not} follow directly from Arveson's conjecture and universality of the $d$-shift, because $T$ is a row contraction of infinite rank. 
\end{remark}



\section{Essential von Neumann inequality on subvarieties}

Recall that a commuting row contraction $T = (T_1, \ldots, T_d)$ is said to be of {\em finite rank} if $\rank(I - \sum_{i=1}^d T_iT_i^* ) < \infty$. 

\begin{theorem}\label{thm:ess_vN}
Let $T = (T_1, \ldots, T_d)$ be a commuting row contraction of finite rank. Assume that $I \triangleleft \mb{C}[z]$ is a homogeneous ideal such that $q(T) = 0$ for all $q \in I$. Then for every matrix valued polynomial $p \in M_n(\mb{C}[z])$
\[
\|p(T)\|_e \leq \sup_{z \in \pV(I)}\|p(z)\|_{M_n} . 
\]
\end{theorem}

It is worth mentioning that the above theorem is interesting even when $T$ satisfies no relation; the supremum on the right hand side is then taken over the entire unit sphere. In that case the result is an application of the results of \cite{Arv98}, along the lines of the following proof. 

\begin{proof}
We continue with the notation set in the introduction. 
Denote $r = \rank (I - \sum_{i=1}^d T_i T_i^*)$. Popescu's constrained dilation theorem \cite[Theorem 2.4]{Popescu06} provides the following:
\begin{enumerate}
\item an $r$-dimensional Hilbert space $G$,
\item a $*$-representation $\sigma$ of $\cT_I$ on some other Hilbert space $H$ which annihilates the compacts,
\item a subspace $L \subseteq (H^2_d \otimes G) \oplus H$ that is invariant under the the operators $(S_i \otimes I_G)^* \oplus \sigma(S_i)^*$, $i=1, \ldots, d$,
\end{enumerate}
such that $T$ is equal to the compression of $(S_i \otimes I_G) \oplus \sigma(S_i)$ to $L$. Since $\sigma$ annihilates the compacts it induces a $*$-representation $\tilde{\sigma}$ of $\cO_I$, thus for every $p \in M_n(\mb{C}[z])$,  $p(T)$ is equal to a compression of 
\[
(p(S) \otimes I_G) \oplus \tilde{\sigma}(p(Z)) \in M_n(B(H^2_d \otimes G \oplus H))
\]
to $M_n(L)$. By Theorem \ref{thm:main}
\[
\|(p(S) \otimes I_G) \oplus \tilde{\sigma}(p(Z))\|_e = \|p(S)\|_e = \sup_{z \in \pV(I)}\|p(z)\|_{M_n}. 
\]
Since compression cannot increase the essential norm, the proof is complete. 
\end{proof}

Let us finish by mentioning a non-trivial class of examples to which the above theorem can be applied. 
\begin{example}
Theorem \ref{thm:ess_vN} provides an alternative route for proving Corollary \ref{cor:Besov_Sobolev}. Indeed, although the row contractions considered there do not have finite rank, they can be compactly perturbed (in a way that does not affect homogeneous relations) to have finite rank. In fact, Equation (\ref{eq:Besov_row_contraction}) shows that this is true also for $\sigma<1/2$.  It is worth mentioning that even for $\sigma \geq (d+1)/2$, when $B^2_\sigma$ is a Bergman type space, the norm of the compression of $p(M_z)$ to $B^2_\sigma \ominus I$ can be strictly bigger than $\|p\|_V$, thus the estimate obtained for the essential norm is not trivial. 
\end{example}

\bibliographystyle{amsplain}

\end{document}